\documentclass[12pt]{amsart}

\usepackage{amsthm,amssymb}
\usepackage{enumitem}

\usepackage{xcolor}

\usepackage{soul}

\usepackage{orcidlink}

\newtheoremstyle{claim}
     {11pt}
     {11pt}
     {}
     {}
     {\itshape}
     {}
     {.5em}
     {\noindent\thmname{#1} \thmnumber{#2}.{\rm\thmnote{#3}}}

\theoremstyle{claim}

\newtheorem{claim}{Claim}

\newtheoremstyle{theorem}
     {11pt}
     {11pt}
     {}
     {}
     {\bfseries}
     {}
     {.5em}
     {\noindent\thmnumber{#2}. \thmname{#1} \thmnote{(#3)}}

\theoremstyle{theorem}

\newtheorem{defi}{Definition}
\newtheorem{thm}[defi]{Theorem}
\newtheorem{coro}[defi]{Corollary}
\newtheorem{lemma}[defi]{Lemma}
\newtheorem{propo}[defi]{Proposition}
\newtheorem{ques}[defi]{Question}
\newtheorem{obs}[defi]{Observation}

\newcommand{\cl}[2]{\mathrm{cl}_{#1}\!\left(#2\right)}

\newcommand{\supp}{\mathsf{supp}}
\newcommand{\ZFC}{\textsf{ZFC}}
\newcommand{\PFA}{\textsf{PFA}}

\title[FC and ES of product spaces and subspaces]{Functional countability and exponential separability of product spaces and subspaces}

\author[R. Hernández-Gutiérrez]{Rodrigo Hernández-Gutiérrez\orcidlink{0000-0002-5949-0871}}
\address[R. Hernández-Gutiérrez]{Departamento de Matemáticas, Universidad Autónoma Metropolitana campus Iztapalapa, Av. San Rafael Atlixco 186, Leyes de Reforma 1a Sección, Iztapalapa, 09310, Mexico city, Mexico}
\email[R. Hernández-Gutiérrez]{rod@xanum.uam.mx}

\author[S. Spadaro]{Santi Spadaro\orcidlink{0000-0001-5880-8799}}
\address[S. Spadaro]{Dipartimento di Matematica e Informatica \\ Universit\` a degli Studi di Catania, viale Andrea Doria, n. 6, 95125, Catania, Italy}
\email[S. Spadaro]{santi.spadaro@unict.it}

\makeatletter
\@namedef{subjclassname@2020}{\textup{2020} Mathematics Subject Classification}
\makeatother

\date{\today}

\begin{document}

\begin{abstract}
We investigate the behavior of functional countability and exponential separability in products and subspaces of topological spaces. We solve a problem of Tkachuk by showing that the product of functionally countable pseudocompact spaces is itself functionally countable. Solving another problem of Tkachuk, we show that it is independent of \textsf{ZFC} whether regular spaces which have all their subspaces functionally countable are hereditarily Lindelöf. Finally, we prove that the $\sigma$-product of non-zero ordinals is exponentially separable, thereby extending a result of Kemoto and Szeptycki. 
\end{abstract}

\maketitle

\section{Introduction}

A space $X$ is \emph{functionally countable} if every continuous real-valued function defined on $X$ has countable image. This notion was first identified (but not named) by Rudin (\cite{rudin-57}) and by Pełczyński and Semadeni (\cite{pel-semadeni}), who independently proved that a compact space is functionally countable if and only if it is scattered. Notable classes of spaces that are functionally countable include Lindelöf $P$-spaces (attributed to A.W. Hager in \cite[Proposition 3.2]{levy-rice}), Lindelöf scattered spaces (\cite[Theorem II.7.13]{arkh-top_funct_spaces}), $\sigma$-products of the discrete space $\{0,1\}$ (mentioned in \cite{levy-matveev-2010}), and scattered suborderable spaces of countable extent (\cite[Corollary 3.6]{tka_wil-fc_GO})

In \cite{galvin-hernandez}, S. Hernández solved a question of Galvin by proving that there exist two functionally countable spaces whose product is not functionally countable. On the other hand, Tkachuk has found conditions under which the product of functionally countable spaces is functionally countable, see \cite{tka-fc_prod-2022}.

In section 2 of this paper we answer a question of Tkachuk's by proving that the finite product of pseudocompact functionally countable spaces is functionally countable (Theorem \ref{thm:product_pseudocompact_FC} and Corollary \ref{coro:finite_product_pseudocompact_FC}).

In \cite{tka-nice_subclass}, Tkachuk introduced the notion of exponentially separable space.

\begin{defi}
Let $X$ be a topological space. Given $\mathcal{F}\subset\wp(X)$, a set $A\subset X$ is \emph{strongly dense} in $\mathcal{F}$ if $A\cap\left(\bigcap\mathcal{G}\right)\neq\emptyset$ whenever $\mathcal{G}\subset\mathcal{F}$ and $\bigcap \mathcal{G}\neq\emptyset$. $X$ is \emph{exponentially separable} if, for every countable family $\mathcal{F}$ of closed sets, there is countable subset $A$ of $X$ that is strongly dense in $\mathcal{F}$.
\end{defi}

Every exponentially separable space is functionally countable, and in fact has all its closed subspaces functionally countable as well \cite[Corollary 3.7]{tka-nice_subclass}. Moreover, functionally countable spaces can be characterized by replacing \emph{closed} with \emph{zero-set} in the definition of exponentially separable, a result that shows how natural this definition is in the investigation of functionally countable spaces.

\begin{thm}
\cite[Theorem 3.8]{tka-nice_subclass} A space $X$ is functionally countable if and only if, for every countable family $\mathcal{F}$ of zero-sets, there is a countable subset $A$ of $X$ that is strongly dense in $\mathcal{F}$.
\end{thm}

From the above theorem it immediately follows (\cite[Corollary 3.9]{tka-nice_subclass}) that a perfectly normal functionally countable space is exponentially separable. Moreover, we essentially know of only two examples of spaces that are functionally countable but not exponentially separable: \cite[Example 3.15]{tka-nice_subclass} and \cite[Theorem 12]{HG_S-2025a}, both of which are not normal. Thus, the following question seems natural.

\begin{ques}\label{ques:normality}
Let $X$ be a normal and functionally countable space. Does it follow that $X$ is exponentially separable?
\end{ques}

In search for an example of a normal functionally countable space that is not exponentially separable, we looked at $\sigma$-products of ordinals. These spaces were shown to be always functionally countable and sometimes normal by Kemoto and Szeptycki in \cite{kem-sz-2004}. Our main result from Section 3 is that these spaces are exponentially separable (Theorem \ref{thm:sigma-prod}). Question \ref{ques:normality} remains open.

Finally, in Section 4 we answer a question of Tkachuk's \cite[Question 4.3]{tka-fc_nice_prod-2025} by showing that it is independent of \ZFC{} whether every regular space all of whose subspaces are functionally countable is hereditarily Lindelöf.

All spaces are assumed to be Hausdorff. Our notation is standard and follows \cite{En} for topological notions and \cite{jech} for set-theoretic notions. A space $X$ is \emph{scattered} if every non-empty subspace of $X$ has an isolated point. Given a set $A$ and a cardinal $\lambda$, the collections $[A]\sp{\lambda}$, $[A]\sp{<\lambda}$ and $[A]\sp{\leq\lambda}$ consist of all subsets of $A$ with cardinality either exactly $\lambda$, less than $\lambda$, or at most $\lambda$, respectively.

\section{Products of pseudocompact spaces}

Recall that a space $X$ is \emph{pseudocompact} if it is Tychonoff and every continuous, real-valued function defined on $X$ is bounded. In \cite[Proposition 3.14]{tka-fc_prod-2022}, Tkachuk proved that the finite product of countably compact and functionally countable spaces is also functionally countable, and asked (\cite[Question 4.7]{tka-fc_prod-2022}) whether countable compactness could be replaced by pseudocompactness in that result. In this section we answer this question on the affirmative, see Theorem \ref{thm:product_pseudocompact_FC} and Corollary \ref{coro:finite_product_pseudocompact_FC} below.

The following two results were proved by Tkachuk.

\begin{thm}\cite[Theorem 3.10]{tka-fc_prod-2022}\label{tka:char_pseudocompact_products}
 Let $X$ and $Y$ be functionally countable spaces. If $X\times Y$ is pseudocompact, then $X\times Y$ is functionally countable.
\end{thm}

\begin{lemma}\cite[Proposition 3.13]{tka-nice_subclass}\label{lemma:pscomp-scattered}
 Every pseudocompact, functionally countable space is scattered.
\end{lemma}

Next, we recall an easy characterization of pseudocompactness for spaces with a dense set of isolated points. For every space $X$, the set of isolated points of $X$ will be denoted by $I(X)$.

\begin{lemma}\cite[Lemma 8.8.5]{hh-hrusak-2018}\label{lemma:pseudocompact_char}
 Let $X$ be a space such that $\cl{X}{I(X)}= X$. The following are equivalent.
 \begin{enumerate}[label=(\roman*)]
  \item $X$ is pseudocompact;
  \item for each $A\in[I(X)]^\omega$ it follows that $\cl{X}{A}\setminus A\neq\emptyset$.
 \end{enumerate}
\end{lemma}

Given a space $X$, $A\in[X]^\omega$ and $p\in X$, $A\rightarrow p$ denotes that $A$ converges to $p$, that is, for every open set $U\subset X$ such that $p\in U$ the set $A\setminus U$ is finite. Recall that a space $X$ is \emph{sequential} if for every non-closed $A\subset X$ there exists $B\in[A]^\omega$ and $p\in X\setminus A$ such that $B\rightarrow p$. In Proposition \ref{propo:existence_many_conv_seq} below, we will prove that functionally countable pseudocompact spaces satisfy a partial version of sequentiality.

Recall that a space $X$ has the weak Whyburn property (see, for example, \cite{spadaro-2016}) if, for every non-closed set $A \subset X$ there is a point $x \in \cl{X}{A} \setminus A$ and a subset $B \subset A$ such that $\cl{X}{B} \setminus B=\{x\}$. The following lemma was first proved by Tkachuk and Yaschenko in \cite[Theorem 2.7]{tka-yash-2001}, where weak Whyburn spaces were referred to as spaces with the property of \emph{weak approximation by points}. We give a simpler proof of it to make the paper self-contained.

\begin{lemma} \label{lemmawhyburn}
Every regular scattered space is weakly Whyburn.
\end{lemma}
\begin{proof}
Let $A \subset X$ be a non-closed set. Let $x$ be an isolated point of $\cl{X}{A} \setminus A$, then there is an neighborhood $V$ of $x$ such that $\cl{X}{V} \cap (\cl{X}{A} \setminus A)=\{x\}$. Set $B=V \cap A$. Then $\cl{X}{B} \setminus A=\{x\}$ and we are done.
\end{proof}

\begin{propo}\label{propo:existence_many_conv_seq}
 Let $X$ be pseudocompact and functionally countable. If $A\in[I(X)]^\omega$, then there exist $p\in X$ and $B\in[A]^\omega$ such that $B\rightarrow p$.
\end{propo}
\begin{proof}
 Assume that $X$ is infinite, otherwise the result is vacuously true. By Lemma \ref{lemma:pscomp-scattered}, $X$ is scattered. Then $I(X)$  is infinite and dense. So let $A\in[I(X)]^\omega$; notice that by Lemma \ref{lemma:pseudocompact_char}, $A$ is not closed. By Lemma \ref{lemmawhyburn} there exists $B\subset A$ and $p\in X\setminus A$ such that $\cl{X}{B}=B\cup\{p\}$. We claim that in fact $B\rightarrow p$. Otherwise there exists a neighborhood $U$ of $p$ such that $B\setminus U$ is infinite. However, this implies that $B\setminus U$ is an infinite subset of $I(X)$ without limit points, which contradicts Lemma \ref{lemma:pseudocompact_char}.
\end{proof}

Finally, the following result provides an answer to Tkachuk's question.
 
\begin{thm}\label{thm:product_pseudocompact_FC}
 Let $X$ and $Y$ be pseudocompact and functionally countable. Then $X\times Y$ is pseudocompact and hence functionally countable.
\end{thm}
\begin{proof}
 By Lemma \ref{lemma:pscomp-scattered} above, it follows that $X$ and $Y$ are scattered. Notice that $X\times Y$ is also scattered and in fact $I(X\times Y)=I(X)\times I(Y)$. Thus, by Lemma \ref{lemma:pseudocompact_char} it is sufficient to prove that any $A\in[I(X)\times I(Y)]^\omega$ has a limit point in $X\times Y$.
 
 Let $\pi_X\colon X\times Y\to X$ and $\pi_Y\colon X\times Y\to Y$ be the two projections. Since $A$ is infinite, at least one of $\pi_X[A]$ or $\pi_Y[A]$ is infinite; assume that $\pi_X[A]$ is infinite. Since $\pi_X[A]\in[I(X)]^\omega$ by Proposition \ref{propo:existence_many_conv_seq} there exists a faithfully indexed sequence $\{x_n\colon n\in\omega\}\subset \pi_X[A]$ and $p\in X$ such that $\{x_n\colon n\in\omega\}\rightarrow p$. For each $n\in\omega$, let $y_n\in Y$ be such that $\langle x_n,y_n\rangle\in A$.
 
 If the set $\{y_n\colon n\in\omega\}$ is finite, then there is $S\in[\omega]^{\omega}$ and $y\in\{y_n\colon n\in\omega\}$ such that $y=y_n$ for all $n\in S$. It follows then that $\{\langle x_n,y\rangle\colon n\in S\}$ is an infinite subset of $A$ such that $\{\langle x_n,y\rangle\colon n\in S\}\rightarrow \langle p,y\rangle$.
 
 Assume now that $\{y_n\colon n\in\omega\}$ is infinite. Since $\{y_n\colon n\in\omega\}\in[I(Y)]^\omega$, by Proposition \ref{propo:existence_many_conv_seq} there exists an infinite set $S$ and $q\in Y$ such that $\{y_n\colon n\in S\}\rightarrow q$. Since any subsequence of a convergent sequence is convergent to the same limit, we also know that $\{x_n\colon n\in S\}\rightarrow p$. It follows that $\{\langle x_n,y_n\rangle\colon n\in S\}\rightarrow \langle p,q\rangle$.
 
 In both cases, we have shown that there is a subset of $A$ with a limit point. By Lemma \ref{lemma:pseudocompact_char} the space $X\times Y$ is pseudocompact. And by Tkachuk's Theorem \ref{tka:char_pseudocompact_products}, the space $X\times Y$ is functionally countable.
 
\end{proof}

\begin{coro}\label{coro:finite_product_pseudocompact_FC}
 Every finite product of spaces that are pseudocompact and functionally countable is also pseudocompact and functionally countable. 
\end{coro}

\section{Finite products and \texorpdfstring{$\sigma$}{sigma}-products of ordinals}

Let $\lambda$ be a cardinal, $\mathcal{X}=\{X_\alpha\colon\alpha<\lambda\}$ a family of spaces and $p\in\prod_{\alpha<\lambda}{X_\alpha}$. For each $x\in\prod_{\alpha<\lambda}{X_\alpha}$ let $\supp_p(x)=\{\alpha<\lambda\colon x(\alpha)\neq p(\alpha)\}$. The $\sigma$-product of the family $\mathcal{X}$ with respect to $p$ is the set $$\textstyle\sigma\left(\prod_{\alpha<\lambda}{X_\alpha},p\right)=\left\{x\in\prod_{\alpha<\lambda}{X_\alpha}\colon \lvert\supp_p(x)\rvert<\omega\right\}.$$

Ordinals will be considered topological spaces with the order topology. If $\{\gamma_\xi\colon\xi<\lambda\}$ is a set of ordinals, the element of the product $\prod_{\xi<\lambda}{\gamma_\xi}$ that has constant value $0$ will be denoted by $\mathbf{0}$. In \cite{kem-sz-2004}, Kemoto and Szeptycki studied the topology of $\sigma$-products of ordinals and, among other things, obtained the following result.

\begin{thm}\label{thm:kemoto-szeptycki}
 Let $\{\gamma_\xi\colon\xi<\kappa\}$ be a set of ordinals.
 \begin{enumerate}[label=(\roman*)]
  \item\label{thm:kemoto-szeptycki-a}\cite[Corollary 3.8]{kem-sz-2004} $\sigma(\prod_{\xi<\kappa}{\gamma_\xi},\mathbf{0})$ is normal if and only if  $\prod_{\xi\in F}{\gamma_\xi}$ is normal for all $F\in[\kappa]\sp{<\omega}\setminus\{\emptyset\}$.
  \item\label{thm:kemoto-szeptycki-b}\cite[Theorem 3.9]{kem-sz-2004} $\sigma(\prod_{\xi<\kappa}{\gamma_\xi},\mathbf{0})$ is functionally countable.
 \end{enumerate} 
\end{thm}

We will prove a strengthening of Theorem \ref{thm:kemoto-szeptycki}\ref{thm:kemoto-szeptycki-b} in Theorem \ref{thm:sigma-prod}.

\begin{lemma}\label{lemma:product_unctble_cofin_ES}
 Any finite product of ordinals, each of which is either a successor or has uncountable cofinality, is exponentially separable.
\end{lemma}
\begin{proof}
 Given $n\in\omega$, let $\{\gamma_i\colon i\leq n\}$ be a collection of ordinals each of which is either a successor or has uncountable cofinality, let $X=\prod_{i\leq n}\gamma_i$, and let $\{A_n\colon n\in\omega\}$ be a collection of closed subsets of $X$.
 
 Define $\mathcal{F}=\left\{F\in[\omega]\sp{<\omega}\colon {\textstyle\bigcap_{n\in F}A_n\neq\emptyset}\right\}$, which is a countable set. For each $F\in\mathcal{F}$, choose $p_F\in\bigcap_{n\in F}A_n$. Define $$Y=\cl{X}{\{p_F\colon F\in\mathcal{F}\}}.$$
 
 We claim that $Y$ is in fact compact. For each $i\leq n$, we define a successor ordinal $\xi_i$. If $\gamma_i$ is a successor, let $\xi_i=\gamma_i$. Otherwise, $\gamma_i$ has uncountable cofinality and let $\xi_i=\sup\{p_F(i)\colon F\in\mathcal{F}\}+1<\gamma_i$. Then notice that $Y\subset\prod_{i\leq n}\xi_i\subset X$ and $\prod_{i\leq n}\xi_i$ is compact. Thus, $Y$ is compact as well.
 
 Since $Y$ is a compact and scattered space, it is exponentially separable (by \cite[Corollary 3.14]{tka-nice_subclass}). Thus, there exists a countable set $C\subset Y$ that is strongly dense in $\{A_n\cap Y\colon n\in\omega\}$. We claim that $C$ is also strongly dense in $\{A_n\colon n\in\omega\}$.
 
 So let $N\in[\omega]\sp{\leq\omega}$ such that $\bigcap_{n\in N}A_n\neq\emptyset$. Notice that this implies that $\bigcap_{n\in N\cap (m+1)}A_n\neq\emptyset$ for each $m\in\omega$; that is, $N\cap(m+1)\in\mathcal{F}$ for each $m\in\omega$. By the definition of $Y$, we obtain that $$\textstyle Y\cap\left(\bigcap_{n\in N\cap (m+1)}A_n\right)\neq\emptyset$$ for all $m\in\omega$. Since $Y$ is compact, it follows that $\bigcap_{n\in N}(A_n\cap Y)\neq\emptyset$. From the definition of $C$ it follows that $C\cap\left(\bigcap_{n\in N}(A_n\cap Y)\right)\neq\emptyset$; in particular $C\cap\left(\bigcap_{n\in N}A_n\right)\neq\emptyset$. This completes the proof.
\end{proof}

\begin{lemma}\label{lemma:sigma-prod-c1}
 For every collection $\{\gamma_\xi\colon\xi<\lambda\}$ of ordinals, all of which are either successors or have uncountable cofinality, $\sigma(\prod_{\xi<\lambda}{\gamma_\xi},\mathbf{0})$ is exponentially separable.
\end{lemma}
\begin{proof}
 For each $n\in\omega$, we denote
 $$\textstyle
 X_n=\left\{f\in\prod_{\xi<\lambda}{\gamma_\xi}\colon \lvert\supp_{\mathbf{0}}(f)\rvert\leq n\right\}
 $$
 Since the countable union of exponentially separable spaces is exponentially separable (\cite[Proposition 3.2(d)]{tka-nice_subclass}), it is sufficient to prove that $X_n$ is exponentially separable for all $n\in\omega$. 
 
 We will need the following facts.
 \begin{enumerate}[label=(\alph*)]
  \item $X_0=\{\mathbf{0}\}$.
  \item For all $n<\omega$, $X_n$ is closed in $\prod\{\gamma_\xi\colon \xi<\lambda\}$.
 \end{enumerate}

Then $X_0$ is clearly exponentially separable.
Now, inductively assume that $X_m$ is exponentially separable, let us prove that $X_{m+1}$ is also exponentially separable.

Given $t\in[\lambda]\sp{m+1}$, let
$$
S_t=\left\{f\in\prod_{\xi<\lambda}{\gamma_\xi}\colon\supp_{\mathbf{0}}(f)=t\right\}.
$$
Notice that $X_{m+1}\setminus X_m=\bigcup\{S_t\colon t\in[\lambda]\sp{m+1}\}$. Clearly, $S_t$ is homeomorphic to $\prod_{\xi\in F}\gamma_\xi$ so we know that it is exponentially separable by Lemma \ref{lemma:product_unctble_cofin_ES}.

\setcounter{claim}{0}
\begin{claim}\label{X1}
Let $G$ be a closed subset of $X_{m+1}$. Then either the set $\mathcal{I}(G)=\{t\in[\lambda]\sp{m+1}\colon G\cap S_{t}\neq\emptyset\}$ is countable, or $G\cap X_m\neq\emptyset$. 
\end{claim}
\begin{proof}[Proof of Claim \ref{X1}]
Assume that the first condition does not hold. We will prove that the second necessarily holds. By the $\Delta$-system lemma, there exists $p\in[\lambda]\sp{<\omega}$ and $\{t_\alpha\colon\alpha<\omega_1\}\subset\mathcal{I}(G)$ such that $t_\alpha\cap t_\beta=p$ if $\alpha<\beta<\omega_1$. For each $\alpha<\omega_1$, let $f_\alpha\in G\cap S_{t_\alpha}$. We claim there exists $g\in\prod_{\xi<\lambda}{\gamma_\xi}$ that is a limit point of $\{f_\alpha\colon\alpha<\omega_1\}$ and $\supp_{\mathrm{0}}{(g)}\subset p$.

Let $k=\lvert p\rvert$. If $k=0$, then $p=\emptyset$; in this case let $g=\mathbf{0}$, which is clearly a limit point of $\{f_\alpha\colon\alpha<\omega_1\}$.

Now, assume that $k>0$ and let $p=\{\xi_j\colon j\leq k\}$. In $k-1$ steps, choose a decreasing collection of sets $\{A_j\colon j\leq k\}\subset[\omega_1]\sp\omega$ and ordinals $\beta_j\in\kappa_{\xi_j}$ such that $\{f_\alpha(\xi_i)\colon \alpha\in A_j\}\rightarrow\beta_j$ for each $j<k$; this is possible since each of the ordinals in $\{\gamma_{\xi_j}\colon j\leq k\}$ is sequentially compact. Define $g\in\prod_{\xi<\lambda}{\gamma_\xi}$ such that $g(\xi_j)=\beta_j$ for $j\leq k$ and $g(\xi)=0$ if $\xi\notin p$. In this case $g$ is the limit of $\{f_\alpha\colon\alpha\in A_k\}$ and clearly $\supp_{\mathrm{0}}{(g)}\subset p$.

Notice that $\lvert p\rvert\leq m$, which implies that $g\in X_m$. Since $g$ is in the closure of $\{f_\alpha\colon\alpha<\omega_1\}\subset G$ and $G$ is closed, then $g\in G$. Thus, $X_m\cap G\neq\emptyset$ and we have proved Claim \ref{X1}.
\renewcommand{\qedsymbol}{$\triangle$}
\end{proof}

\begin{claim}\label{X2}
Let $\{G_i\colon i<\omega\}$ be a collection of closed subsets of $X_{m+1}$ such that $G_{i+1}\subset G_i$ for all $i<\omega$. If $G_i\cap X_m\neq\emptyset$ for all $i<\omega$, then $(\bigcap_{i<\omega}{G_i})\cap X_m\neq\emptyset$. 
\end{claim}
\begin{proof}[Proof of Claim \ref{X2}]
For each $i<\omega$, let $f_i\in G_i\cap X_m$. Let $$A=\bigcup\{\supp_{\mathbf{0}}{(f_i)}\colon i<\omega\},$$ which is a countable set. 

Assume for a moment that $A$ is infinite. Then we may find an enumeration $A=\{\xi_j\colon j\in\omega\}$. We recursively construct a decreasing sequence $\{N_j\colon j\in\omega\}\subset[\omega]\sp\omega$ such that $N_0=\omega$, and a sequence of ordinals $\{\beta_j\colon j\in\omega\}$ as follows. Given $j\in\omega$, $\{f_i(\xi_j)\colon i\in N_j\}$ is a countable subset of the sequentially compact ordinal $\gamma_{\xi_j}$ so we may choose $N_{j+1}\in [N_j]\sp\omega$ and $\beta_j\in\gamma_{\xi_j}$ such that $\{f_i(\xi_j)\colon i\in N_{j+1}\}\rightarrow\beta_j$. Next, construct a pseudointersection $N=\{k_j\colon j\in\omega\}\in[\omega]\sp\omega$ of the sequence $\{N_j\colon j\in\omega\}$ as usual: $k_0=0$ and $k_{j+1}=\min\{i\in N_{j+1}\colon i>k_j\}$. Then $\{f_i(\xi_j)\colon i\in N\}\rightarrow\beta_j$ for all $j<\omega$. Define $g\in\prod_{\xi<\lambda}{\gamma_\xi}$ as follows: $g(\xi_j)=\beta_j$ if $j<\omega$ and $g(\xi)=0$ if $\xi\notin A$. Then $g$ is clearly a limit point of $\{f_i\colon i<\omega\}$.

If $A$ is finite, we can certainly proceed in a analogous way (but with only finitely many steps) to construct $g\in\prod_{\xi<\lambda}{\gamma_\xi}$ that is a limit point of $\{f_i\colon i<\omega\}$.

In any of the two cases, since $G\cap X_m$ is closed in  $\prod_{\xi<\lambda}{\gamma_\xi}$ for every $m<\omega$, it follows that $g\in(\bigcap_{i<\omega}{G_i})\cap X_m$.
\renewcommand{\qedsymbol}{$\triangle$}
\end{proof}

Let $\{F_i\colon i\in\omega\}$ be a family of closed subsets of $X_{m+1}$. 

Since we are inductively assuming that $X_m$ is exponentially separable, there exists a countable set $C\subset X_m$ that is strongly dense in $\{F_i\cap X_m\colon i\in\omega\}$.

For each $s\in[\omega]\sp{<\omega}$, define $F_s=\bigcap\{F_i\colon i\in s\}$, which is a closed set. Let $A=\{s\in[\omega]\sp{<\omega}\colon F_s\neq\emptyset,\, F_s\cap X_m=\emptyset\}$. For each $s\in A$, by Claim \ref{X1}, $\mathcal{I}(F_s)$ is countable. Then $\mathcal{I}=\bigcup\{\mathcal{I}(F_s)\colon s\in A\}$ is countable. 

For each $t\in\mathcal{I}$, $S_t$ is homeomorphic to $\prod_{i\in t}{\gamma_i}$, which is exponentially separable by Lemma \ref{lemma:product_unctble_cofin_ES}. Thus, there exists a countable set $C_t\subset S_t\subset X_{m+1}$ that is strongly dense in $\{F_n\cap S_t\colon n\in \omega\}$.

Finally, we define $D=C\cup(\bigcup\{C_t\colon t\in\mathcal{I}\})$. Clearly, $D$ is countable, we claim that it is strongly dense in $\{F_i\colon i\in\omega\}$. 

Let $Z\subset\omega$ and assume that $\bigcap_{i\in Z}{F_i}\neq\emptyset$. We have to prove that $(\bigcap_{i\in Z}{F_i})\cap D\neq\emptyset$.

First, assume that $(\bigcap_{i\in Z}{F_i})\cap X_m\neq\emptyset$. By the definition of $C$, it follows that $(\bigcap_{i\in Z}{(F_i\cap X_m)})\cap C\neq\emptyset$. In particular, $(\bigcap_{i\in Z}{F_i})\cap D\neq\emptyset$.

Now, assume that $(\bigcap_{i\in Z}{F_i})\cap X_m=\emptyset$. Let $\varphi\colon\omega\to\omega$ be a non-decreasing function such that $Z=\mathsf{im}(\varphi)$ (which might not be injective when $Z$ is finite). By the contrapositive statement of Claim \ref{X2}, there exists $k<\omega$ such that $(\bigcap_{j\leq k}{F_{\varphi(j)}})\cap X_m=\emptyset$. Let $r=\{\varphi(j)\colon j\leq k\}$, then $r\in A$.

Since $r\in A$, by the definition of $\mathcal{I}(F_r)$, $F_r\subset\bigcup_{t\in\mathcal{I}(F_r)}{S_t}$. Thus, $\bigcap_{i\in Z}{F_i}\subset\bigcup_{t\in\mathcal{I}(F_r)}{S_t}$. There must exist some $u\in\mathcal{I}(F_r)$ such that $(\bigcap_{i\in Z}{F_i})\cap S_u\neq\emptyset$. By the definition of $C_u$, $(\bigcap_{i\in Z}{(F_i\cap S_u)})\cap C_u\neq\emptyset$. In particular, $(\bigcap_{i\in Z}{F_i})\cap D\neq\emptyset$.

This completes the proof that $X_{m+1}$ is exponentially separable. Thus, we have completed the recursion and it follows that $\sigma(\prod_{\xi<\lambda}{\gamma_\xi},\mathbf{0})$ is exponentially separable.
\end{proof}

\begin{thm}\label{thm:sigma-prod}
Let $\{\gamma_\xi\colon\xi<\lambda\}$ be a set of non-zero ordinals. Then $\sigma(\prod_{\xi<\lambda}{\gamma_\xi},\mathbf{0})$ is exponentially separable.
\end{thm}
\begin{proof}
Let $\{\gamma_\xi\colon \xi<\lambda\}$ be a collection of non-zero ordinals.

Let $A$ be the set of all $\xi<\lambda$ such that $\gamma_\xi$ is either a successor or has uncountable cofinality, and let $B=\lambda\setminus A$.

For each $\xi\in B$, we may assume that $\gamma_\xi=\sup\{\beta(\xi,j)\colon j\in\omega\}$, where $\beta(\xi,j)$ is a successor and $\beta(\xi,j)<\beta(\xi,j+1)$ for all $j\in\omega$. If $\xi\in A$, define $\beta(\xi,j)=\gamma_\xi$ for all $j\in\omega$.

For each $j\in\omega$, let $X_j=\sigma(\prod_{\xi<\lambda}{\beta(\xi,j)},\mathbf{0})$. Notice that $X_j$ is a $\sigma$-product of ordinals, each of which is either a successor or has uncountable cofinality, so by Lemma \ref{lemma:sigma-prod-c1}, $X_j$ is exponentially separable. Clearly, $\sigma(\prod_{\xi<\lambda}{\gamma_\xi},\mathbf{0})=\bigcup_{j\in\omega}{X_j}$ so it is exponentially separable.
\end{proof}

Since $\sigma$-products of finitely many factors are just products, we obtain the following corollary.

\begin{coro}
 The product of any finite collection of non-zero ordinals is exponentially separable.
\end{coro}

\section{Hereditarily Lindelöf spaces}

We will say that a space $X$ is \emph{hereditarily functionally countable} if every subspace of $X$ is functionally countable. 
In \cite[Corollary 3.9]{tka-fc_nice_prod-2025}, Tkachuk proved that if a Tychonoff space is hereditarily Lindelöf and functionally countable, then it is necessarily hereditarily functionally countable. In fact, the following is true.

\begin{obs}
Let $X$ be a Tychonoff and hereditarily Lindelöf space. If $X$ is functionally countable then every subspace of $X$ is exponentially separable.
\end{obs}
\begin{proof}
 Let $Y\subset X$. Since $X$ is regular and hereditarily Lindelöf, $X$ is perfectly normal. In particular, $Y$ is perfectly normal. By \cite[Corollary 3.9]{tka-fc_nice_prod-2025}, $Y$ is functionally countable. In \cite[Corollary 3.9]{tka-nice_subclass} it is proved that perfectly normal, functionally countable spaces are exponentially separable. Thus, $Y$ is exponentially separable.
\end{proof}

Thus, in \cite[Question 4.3]{tka-fc_nice_prod-2025} Tkachuk asked whether every Tychonoff and hereditarily functionally countable space is hereditarily Lindelöf. We will answer this question. First, we prove that an Ostaszewski space is an example of a hereditarily functionally countable space that is not Lindelöf. We will use the following definition of Ostaszewski space, taken from \cite{eisworth-roitman-1999}.

A Hausdorff space $X$ is \emph{sub-Ostaszewski} if it is uncountable and every closed set of $X$ is either countable or cocountable. An \emph{Ostaszewski space} is defined to be a regular, countably compact and non-compact sub-Ostaszewski space. Ostaszewski's classic space, constructed in \cite{ostaszewski} from $\diamondsuit$, is sub-Ostaszewski (see also \cite[5.3]{roitman-handbook}).

\begin{thm}\label{thm:subOst-her_es}
 If $X$ is a sub-Ostaszewski space, then any $Y\subset X$ is exponentially separable.
\end{thm}
\begin{proof}
 According to \cite[Proposition 2.10]{eisworth-roitman-1999} we may assume that the underlying set of $X$ is $\omega_1$.
 
 Let $Y\subset\omega_1$ and let $\mathcal{F}=\{F_n\colon n\in\omega\}$ be a collection of closed subsets of $Y$. Define $A=\{n\in\omega\colon F_n\textrm{ is uncountable}\}$ and $B=\omega\setminus A$. Notice that if $A=\emptyset$, then $\bigcup_{n\in\omega}F_n$ is a countable set that is strongly dense in $\mathcal{F}$. Thus, we will now assume that $A\neq\emptyset$.
 
 Let $n\in A$. Since $F_n$ is uncountable, there exists $\alpha_n\in\omega_1$ such that $\omega_1\setminus\alpha_n\subset\cl{X}{F_n}$. But $F_n$ is closed in $Y$, so  $Y\setminus\alpha_n\subset Y\cap \cl{X}{F_n}=F_n$. Let $\alpha\in\omega_1$ be such that $\alpha\geq\sup\{\alpha_n\colon n\in A\}$, $\alpha\in Y$ and $\bigcup_{n\in B}F_n\subset\alpha$. Then $(\alpha+1)\cap Y$ is a countable subset of $Y$ that is strongly dense in $\mathcal{F}$.
\end{proof}

Since regular sub-Ostaszewski spaces are $S$-spaces (this follows from \cite[Proposition 2.4]{eisworth-roitman-1999}, \cite[Proposition 2.6]{eisworth-roitman-1999} and \cite[Proposition 2.10]{eisworth-roitman-1999}), regular sub-Ostaszewski spaces do not exist in all models of \ZFC. In fact, we have the following corollary.

\begin{coro}
 It is independent of \ZFC{} whether every regular hereditarily functionally countable space is hereditarily Lindelöf.
\end{coro}
\begin{proof}
 It is well-known that Ostaszewski spaces exist under $\diamondsuit$, and these are regular, hereditarily functionally countable by Theorem \ref{thm:subOst-her_es}, but not Lindelöf.
 
 Now, work in a model of \PFA{} and assume that $X$ is not hereditarily Lindelöf. By \cite[Theorem 3.1]{roitman-handbook}, $X$ contains an uncountable right-separated subspace $Y$. Since \PFA{} implies there are no $S$-spaces by Todorčević's celebrated theorem (see \cite[Theorem 8.9]{todorcevic-partition_problems}), from \cite[Proposition 3.3]{roitman-handbook} it follows that $Y$ contains an uncountable discrete subspace. But then $X$ is contains a subspace that is not functionally countable. Thus, \PFA{} implies that all hereditarily functionally countable spaces are hereditarily Lindelöf.
\end{proof}

\section*{Acknowledgements}

The research of the first-named author was supported by the 2025 CIMPA-ICTP ``Research in Pairs'' fellowship program and the grant CBF2023-2024-2028 from the SECIHTI.

\end{document}